\documentclass{amsart}
\usepackage{amsmath}
\usepackage[abbrev]{amsrefs}
\usepackage{amsfonts}
\usepackage{amsthm}
\usepackage{tikz}
\usepackage[english]{babel}
\usepackage[a4paper, total={6in, 8in}]{geometry}
\usepackage{color}
\usepackage{bbm}
\usepackage{hyperref}
\usepackage{mathrsfs}
\usepackage{graphicx}
\usepackage[symbol]{footmisc}

\usepackage[utf8x]{inputenc}

\def \tp{\tilde{p}}
\def \R{\mathbb{R}}

\def\pw{p_D^{(w)}}
\def\tpw{\tilde{p}^{(w)}}

\newtheorem{theorem}{Theorem}[section]
\newtheorem{lemma}[theorem]{Lemma}
\newtheorem{proposal}[theorem]{Proposition}

\newtheorem{corollary}[theorem]{Corollary}

\usepackage{mathtools}

\begin{document}

\title[Heat kernel estimates of fractional Schrödinger operators]{Heat kernel estimates of fractional Schrödinger operators with Hardy potential on half-line}
\author{Tomasz Jakubowski}
\address{Wroc\l{}aw University of Science and Technology}
\email{tomasz.jakubowski@pwr.edu.pl}
\author{Pawe\l{} Maciocha}
\address{Wroc\l{}aw University of Science and Technology}
\email{pawel.maciocha@pwr.edu.pl}


\begin{abstract}
We provide sharp two-sided estimates of the heat kernel of the Dirichlet fractional Laplacian on the half-line perturbed by the Hardy potential.

\end{abstract}

\maketitle

\section{Introduction}
	Let $\alpha \in (0,2)$ and $D=(0,\infty) \subset \R.$ We consider the following Schr{\"o}dinger operator on $D$ with the Hardy potential, 
	\begin{equation}\label{eq:SchrOp}
		\Delta_D^{\alpha/2} + \kappa x^{-\alpha}.
	\end{equation}
	Here, $\Delta_D^{\alpha/2}$ is the fractional Laplacian with Dirichlet condition on the half line and $\kappa x^{-\alpha}$ is a perturbing Hardy potential. A positive number $\kappa$ is defined as follows, 
	\begin{equation*} \label{eq:kappadelta}
		\kappa := \kappa_{\delta} = \frac{\Gamma(\delta+\alpha/2)\Gamma(1-\delta+\alpha/2)}{\Gamma(\delta)\Gamma(1-\delta)}, \qquad \delta \in (0,1).
	\end{equation*}
	
	The main purpose of this article is to investigate the properties of the heat kernel $\tp(t,x,y)$ of the operator $L \coloneqq \Delta_{D}^{-\alpha/2} + q(x),$ where $q(x) = \kappa x^{-\alpha}.$ 
	Our main theorem is as follows. 
	\begin{theorem}\label{thm:main}
		For $\delta \in (0,1/2]$, the Schr\"odinger operator \eqref{eq:SchrOp} has the heat kernel $\tilde{p}(t,x,y)$, which is jointly continuous on $(0,\infty)\times D \times D$, and satisfies the following estimates
		\begin{align}\label{eq:mainThmEst}
			\tilde{p}(t,x,y) \approx  \left( 1 \wedge \frac{x}{t^{1/\alpha}} \right)^{\alpha/2 - \delta} \left( 1 \wedge \frac{y}{t^{1/\alpha}} \right)^{\alpha/2 - \delta}  \left( t^{-1/\alpha}\wedge \frac{t}{|x-y|^{1+\alpha}} \right), \quad x,y \in D,\ t>0.
		\end{align}
	\end{theorem}
In \cite{MR2457489} it was shown that if the potential $q(x)$ belongs to the so-called Kato class, $\tilde{p}$ is comparable with the unperturbed heat kernel. This is the case, for example if $q(x) = \kappa x^{-\beta}$ with $\beta \in (0,\alpha)$. In this context, the potential $q(x) = \kappa x^{-\alpha}$ is critical since it does not belong to the Kato class. Consequently, the corresponding perturbed density $\tilde{p}$ is not comparable with the unperturbed density $p_D$, see \eqref{eq:pApprox} and \eqref{eq:pDApprox}. 

	\indent We write $f \approx g$ if $f, g \ge 0$ and $c^{-1}g \le f \le cg$ for some positive number $c$. By $c, c_i, C$ we denote constants whose exact values are unimportant, these constants are determined anew in each statement and proof. As usual we write $a \wedge b := \min(a,b)$ and $a \vee b := \max(a,b).$

\subsection{Historical background}	
	Critical perturbations by Hardy potentials were studied for the first time by Baras and Goldstein \cite{MR742415} in the case ${\alpha = 2}$. They proved that there exists a non-trivial nonnegative solution of the heat equation $\partial_t = \Delta + \kappa |x|^{-2}$ in $\mathbb{R}^d$ for $0 \leq \kappa \leq (d-2)^2/4$ and that such a solution does not exist for larger constant $\kappa$. We would like to point out that in the case $d = 1$, the above equation is considered on the half-line. This operator was also studied by Vazquez and Zuazua \cite{VAZQUEZ2000103} in bounded subsets of $\mathbb{R}^d$ as well as in the entire space. Sharp upper and lower estimates of the heat kernel of the Schr\"odinger operator were obtained by Liskevich and Sobol \cite{Liskevich2003} for $0 < \kappa < (d-2)^2/4,$ and by Milman and Semenov for $\kappa \leq (d-2)^2/4$ see \cite[Theorem 1]{MILMAN2004373} and \cite{MILMAN2005238}. These results were proved by another method in the more general case by G. Metafune, L. Negro and C. Spina in \cite{MR3820411}. The authors considered the operator $L = \Delta + (a-1)\sum_{i,j=1}^d \frac{x_ix_j}{|x|^2} D_{ij} + c\frac{x}{|x|^2} \cdot \nabla - \frac{b}{|x|^2}$ and proved two sided sharp estimates of the heat kernel. In particular, they give the exact formula for the heat kernel of the operator $Lu= u_{xx}  + \frac{b}{x^2}u$ on the half-line
\begin{align*}
\tilde{g}(t,x,y) = \frac{1}{2t}\sqrt{xy} I_\nu\Big(\frac{xy}{2t}\Big) \exp\left[-\frac{x^2+y^2}{4t}\right], \qquad t,x,y>0.
\end{align*}	
Here, $\nu = \sqrt{1/4 - b}$, $b\le1/4$ and $I_\nu$ is the modified Bessel function of the first kind. By using the estimates of $I_\nu$ they get
\begin{align*}
\tilde{g}(t,x,y) &\le C(\epsilon) \frac{1}{t}\left( 1 \wedge \frac{x}{t^{1/2}} \right)^{1/2 + \nu} \left( 1 \wedge \frac{y}{t^{1/2}} \right)^{1/2 - \nu} \exp\left[-(1-\epsilon)\frac{|x-y|^2}{4t}\right],\\
\tilde{g}(t,x,y) &\ge C\frac{1}{t}\left( 1 \wedge \frac{x}{t^{1/2}} \right)^{1/2 + \nu} \left( 1 \wedge \frac{y}{t^{1/2}} \right)^{1/2 - \nu} \exp\left[-\frac{|x-y|^2}{4t}\right],
\end{align*}
where $C(\epsilon) \to\infty$ when $\epsilon\to0$. In this context, Theorem \ref{thm:main} may be treated as a fractional counterpart of the estimates above.

	Perturbations by Hardy potentials were considered also for non-local operators. In  \cite{MR3479207, MR3492734} it was shown that the heat kernel of the operator $L \coloneqq \Delta^{\alpha/2} + \kappa |x|^{-\alpha}$ in $\R^d$ does not exist when the constant $\kappa > \kappa^{*} \coloneqq \frac{2^{\alpha}\Gamma((d + \alpha)/4)^2}{\Gamma((d-\alpha)/4)^2)}$ and $\alpha \in (0, d \wedge 2).$ For $\kappa \in (-\infty,\kappa^*]$ the following sharp estimates for the heat kernel $\tp(t, x, y)$ of $L$ were obtained by Bogdan et al. \cite{MR3933622}, Jakubowski and Wang \cite{MR4140086}, Cho et al. \cite{MR4163128},
	\begin{equation} \label{eq:tpr}
		\tp(t,x,y) \approx \left( 1 \land \frac{|x|}{t^{1/\alpha}} \right)^{-\delta} \left( 1 \land \frac{|y|}{t^{1/\alpha}} \right)^{-\delta}  \left( t^{-d/\alpha} \wedge \frac{t}{|x-y|^{d+ \alpha}}\right), \qquad t > 0,\, x,y \in \mathbb{R}^d\setminus \{0\}.
	\end{equation} 
Here, $\delta \in (-\alpha, (d-\alpha)/2]$ is related to $\kappa$ by the formula 
	\begin{equation*}
	\kappa =	\kappa_{\mathbb{R}^d} = \frac{2^{\alpha}\Gamma \left( \frac{\delta + \alpha}{2} \right) \Gamma \left( \frac{d - \delta}{2} \right) }{ \Gamma \left( \frac{\delta}{2} \right) \Gamma \left( \frac{d - \delta - \alpha}{2} \right) }.
	\end{equation*}	
The estimates \eqref{eq:tpr} were a key ingredient in the analysis of Sobolev norms by Frank et al. \cite{MR4206613}, Merz \cite{MR4311597}, and Bui and D'Ancona \cite{MR4521941}. It is also worth to mention the recent papers \cite{jakubowski2023relativistic} and \cite{TJ-KK-KS-2023}, where the heat kernels of the relativistic operators with Hardy potentials were considered.

	\subsection{Comments}
The topic of the paper is closely related to theory of Hardy inequalities. The construction of the potential $q(x)=\kappa_\delta x^{-\alpha}$ in \eqref{eq:SchrOp} is based on the method proposed in the paper \cite{MR3460023} (see Section \ref{sec:3} for some more details).  Such construction permits to get Hardy-type identity called also the ground state representation (see, e.g. \cite{MR3460023}, \cite{jakubowski2022groundstate}, \cite{bogdan2023ground}). It turns out that the maximal value of $\kappa_\delta$ is the best constant in the corresponding Hardy inequality (see \cite{MR3460023}, \cite{jakubowski2022groundstate}, \cite{bogdan2023ground}). We refer also for other results in this direction \cite{MR2663757}, \cite{FRANK20083407}, \cite{MR2425175}, \cite{KijaczkoLenczewska}.

We note that if $\kappa$ in \eqref{eq:SchrOp} is negative, the estimates of $\tp$ are already known. Namely,
\begin{align}\label{eq:mainThmnegat}
			\tilde{p}(t,x,y) \approx  \left( 1 \wedge \frac{x}{t^{1/\alpha}} \right)^{\alpha/2 + \gamma} \left( 1 \wedge \frac{y}{t^{1/\alpha}} \right)^{\alpha/2 + \gamma}  \left( t^{-1/\alpha}\wedge \frac{t}{|x-y|^{1+\alpha}} \right), \quad x,y \in D,\ t>0.
		\end{align}
where $\gamma>0$ is a constant depending on $\kappa$. Inequality \eqref{eq:mainThmnegat} is a special case of the main result of \cite{MR4163128} proved for a wide class of domains $D \subset \R^d$. In particular, the estimates of the form \eqref{eq:mainThmnegat} holds if $D$ is a half-space. In \cite{MR4634216} and \cite{bui2023equivalence}, T.A. Bui, R. Frank and K. Merz used these estimates to derive some inequalities involving the Sobolev norms of the operators $\Delta^{\alpha/2} + \kappa x_d^{-\alpha}$ with $\kappa<0$ in the half-space $\{x=(x_1,\ldots,x_d) \colon x_d>0\}$. According to the remarks made by the authors of these papers, by Theorem \ref{thm:main}, for $d=1$ their results extend to the case $\kappa \in (0, \Gamma(\tfrac{\alpha+1}{2})^2/\pi]$. 
	
In \cite{MR3000465}, the authors presented a quite general method, which allows to estimate heat kernel of the operators perturbed by potentials from the Kato class.  However, the analysis of the heat kernel corresponding to the operator with the so-called critical potential requires a different approach. In this paper, we use the methods developed in \cite{MR3933622}, nevertheless we make some adjustment and improvements. The Theorem \ref{thm:tpInt}, which is crucial in our consideration is proved in slightly simpler method. Additionally, in Lemma \ref{lemma:Inttp(z_-beta)Approx}  we correct a minor error in the formulation and the proof of \cite[Lemma 4.5]{MR3933622}. The other adjustments are due to the fact that the heat kernel $p_D(t,x,y)$ of $\Delta_D^{\alpha/2}$ is subprobabilistic. To deal with this problem, we use the Doobs conditioning and consider the weighted density $p_D^{(w)}(t,x,y)$ = $x^{-\alpha/2}p_D(t,x,y)y^{\alpha/2}$, however, the lack of symmetry of the kernel $p_D^{(w)}$ leads to other difficulties.

Another issue is the range of the parameters $\alpha$ and $\delta$ in the estimates \eqref{eq:mainThmEst}. In \cite{MR3933622}, the operator $L$ was considered in $\R^d$ and it was assumed that $\alpha< d \land 2$ and the critical value $\kappa$ was attained for $\delta= (d-\alpha)/2$. In our case we allow $\alpha \in (0,2)$ and $\delta \in (0,1/2]$, which resembles the case $\alpha=2$, where $\kappa_\delta = \delta(1-\delta)$ attains its maximum at $1/2$.

The article is organized as follows. Section 2 is divided into two parts. In Subsection 2.1, we introduce the concept of a killed process and give some properties of its transistion density $p_D(t,x,y)$. In Subsection 2.2, by the use of the perturbations series, we construct the density $\tp(t,x,y)$ of a perturbed semigroup. Then, we describe its most important properties from our point of view. In Section 3 we analize auxiliary integrals, e.g. those of the form $\int_D \tp(t,x,y) y^{\alpha/2- \beta} dy$. The main results of this section are Theorem \ref{thm:tpInt} and Lemma \ref{lemma:tpIntApprox}. Finally, in Chapter 4 we provide lower and upper estimates of $\tp(t,x,y)$.   
	

\section{Preliminaries}\label{sec:2}
Below, we present some basic	facts about the densities of isotropic $\alpha$-stable processes killed on exiting the half-line $(0,\infty)$. Let 
	\begin{equation*}
			\nu(y) = \frac{\alpha \Gamma(\alpha)\sin\big(\pi\frac{\alpha}{2}\big)}{\pi}  |y|^{-1-\alpha}.
	\end{equation*}
	The fractional Laplacian for (smooth compactly supported) test functions  $\phi \in C_c^\infty(\R)$ is
	\begin{align*}
		\Delta^{\alpha/2} \phi(x) = \lim_{\varepsilon\downarrow 0} \int_{B(0,\varepsilon)^c}(\phi(x+y)-\phi(x)) \nu(y) dy, \quad x \in \R.
	\end{align*}
We consider the convolution semigroup of functions
\begin{align*}
p_t(x) = \frac{1}{2\pi} \int_{\R} e^{ixz} e^{-t|z|^\alpha} dz, \qquad t>0, x \in \R.
\end{align*}
	It is well known (see e.g. \cite{MR2569321}) that  $p(t,x,y)= p_t(x-y)$ is a transition density function of a rotationally symmetric $\alpha$-stable L\'evy process $(X_t)_{t \ge 0}$ whose generator is the fractional Laplacian $\Delta^{\alpha/2}$. For $p(t,x,y)$ the following approximation holds (\cite{MR2569321}) 
	\begin{equation}\label{eq:pApprox}
		p(t,x,y) \approx t^{-1/\alpha} \wedge \frac{t}{|x-y|^{1 + \alpha} }, \quad t >0,~x,y \in \R.
	\end{equation}
	Moreover $p(t,x,y)$ enjoys the scaling property
	\begin{equation*}\label{eq:pScal}
		p(t,x,y) = t^{-1/\alpha}p(1, t^{-1/\alpha}x, t^{-1/\alpha}y), \quad t > 0,~x,y \in \R,
	\end{equation*}
	and is symmetric in the sense that $p(t,x,y) = p(t,y,x)$ for $t > 0$ and $x,y \in \R.$  Clearly, $p(t,x,y)$ satisfies Chapman-Kolmogorov equation 
	\begin{equation*}\label{eq:pCK}
		p(t+s,x,y) = \int_{\mathbb{\R}} p(t,x,z)p(s,z,y)dz, \quad s, t > 0,~x,y \in \R.
	\end{equation*}
	
	\subsection{Killed process} We define the time of the first exit of the process $X_t$ from $D$ by 
	\begin{equation*}
		\tau_D = \inf\{ t \ge 0\colon X_t \in D^c \}.
	\end{equation*}
	The semigroup generated by the process $X_t$ killed on exiting $D$ is determined by transition densities $p_D(t,x,y)$ given by the Hunt formula (see e.g. \cite{MR1329992})
	\begin{equation*} \label{eq:pD}
		p_D(t,x,y) = p(t,x,y) - \mathbb{E}^x[ p(t-\tau_D, X_{  \tau_D}, y) \mathbf{1}_{ \{ \tau_D < t \} } ],\qquad t>0,~x,y\in D.
	\end{equation*}
	The kernel $p_D$ is symmetric in the sense that $p_D(t,x,y) = p_D(t,y,x)$, satisfies scaling property and Chapman-Kolmogorov equation \cite{MR2569321}
	\begin{align*} 
		p_D(t,x,y) &= t^{-1/\alpha} p_D(1, t^{-1/\alpha}x, t^{-1/\alpha}y), \quad t > 0,~x,y \in D, \\ 
		p_D(t+s,x,y) &= \int_{D} p_D(t,x,z)p_D(s,z,y)dz, \quad s, t > 0,~x,y \in D. 
	\end{align*}
The function $p_D$ is the heat kernel of the fractional Laplacian with the Dirichlet condition on $D$ 
	\begin{align*}
		\Delta^{\alpha/2}_D \phi(x) = \lim_{\varepsilon\downarrow 0} \int_{B(0,\varepsilon)^c}(\phi(x+y)-\phi(x)) \nu(y) dy, \quad x \in \R, \; \phi \in C_c^\infty(D).
	\end{align*}
It is known that (see. e.g. \cite{MR2722789})
	\begin{equation} \label{eq:pDApprox}
		p_D(t,x,y) \approx \bigg(1 \wedge \frac{x^{\alpha/2}}{t^{1/2}} \bigg) \bigg(1 \wedge \frac{y^{\alpha/2}}{t^{1/2}} \bigg) p(t,x,y), \quad t >0,~x,y \in D. 
	\end{equation}
In what follows we will repeatedly multiply $p_D$ by $y^{\alpha/2}$. For that reason, we introduce weighted transition density $\pw$:
	\begin{equation*}\label{eq:pw}
		\pw(t,x,y) = \frac{p_D(t,x,y)}{x^{\alpha/2}}y^{\alpha/2}, \qquad t>0,~x,y \in D.
	\end{equation*}
	Due to \eqref{eq:pwIntegral} $\pw$ is probabilistic with respect to the Lebesgue measure. 
	It can be easily checked that  $\pw$ preserves all properties mentioned for $p_D$ except the symmetry. 	Nevertheless, we have
	\begin{equation} \label{eq:pwQuSym}
		\pw(t,x,y) =\pw(t,y,x)  \frac{y^{\alpha}}{x^{\alpha}}, \qquad t>0,~x,y \in D.
	\end{equation}
 	
Throughout the paper, we will frequently refer to the following inequality, which is a direct consequence of \eqref{eq:pDApprox}:
	\begin{equation} \label{eq:pwUsefullApprox}
		\pw(t,x,y) \leq c\,\frac{t}{x^{\alpha/2} y^{1 + \alpha/2}} \bigg(1 \wedge \frac{x^{\alpha/2}}{t^{1/2}} \bigg) \bigg(1 \wedge \frac{y^{\alpha/2}}{t^{1/2}} \bigg), \qquad t >0,~y > 2x. 
	\end{equation}
	
\subsection{Schr{\"o}diger perturbations by Hardy potential}\label{sec:3}		
In this section we define the function $\tilde{p}$, the density of the semigroup generated by the operator  $\Delta_D^{\alpha/2} + \kappa x^{-\alpha}$. In the construction we follow the ideas and results from \cite{MR3460023} and \cite{jakubowski2022groundstate}.	For $\beta \in (0,1)$ and $\gamma \in (\beta + \alpha/2, 1 + \alpha/2)$, let
	\begin{equation*}
	f(t) = \begin{cases}
	C t^{ (-\alpha/2 - \beta + \gamma)/\alpha },\qquad &\text{for}~t>0\\
	0,\qquad &\text{for}~t\leq0.
	\end{cases}
	\end{equation*}	
	Here positive constant $C$ is chosen so that \cite[Section~3]{jakubowski2022groundstate} 
	\begin{equation} \label{def:h}
		h_{\beta}(x) := \int_{0}^{\infty} \int_{D} \pw(t,x,y)f(t)y^{-\gamma - \alpha/2} dydt = x^{- \beta}, \qquad x \in D.
	\end{equation}
	We define
	\begin{equation} \label{def:q}
		q_{\beta}(x) := \frac{1}{h_{\beta}(x)} \int_{0}^{\infty} \int_{D} 	\pw(t,x,y)f'(t)y^{-\gamma - \alpha/2} dydt = \kappa_{\beta}  x^{-\alpha},\qquad x \in D,
	\end{equation}
	where (see \cite{jakubowski2022groundstate}*{Theorem 1.1})
	\begin{align}\label{eq:kappabeta}
		\kappa_\beta = \frac{\Gamma(\beta+\alpha/2)\Gamma(1-\beta+\alpha/2)}{\Gamma(\beta)\Gamma(1-\beta)}.
	\end{align}
The function $\beta \to \kappa_{ \beta }$ is increasing on $\beta \in (0, 1/2)$ and decreasing on $\beta \in (1/2, 1)$ (see. \cite{jakubowski2022groundstate}*{Section 3}). This fact will be important in our integral analysis.	
The following Lemma will be crucial in the proof of  Theorem \ref{thm:tpInt}.
	\begin{lemma} \label{lemma:pwIdentity}
		For $t>0$ and all $x \in D$,
		\begin{equation*} \label{eq:pwIdentity}
			\kappa_{ \beta } \int_{ 0 }^{ t } \int_{ D } \pw(s,x,y) y^{ -\alpha - \beta } dy ds = x^{ - \beta  }  - \int_{ D } \pw(t,x,y)y^{ - \beta } dy.
		\end{equation*}
	\end{lemma}
	\begin{proof}
		We fix $t>0, x>0$ and $\beta \in (0,1)$. By \eqref{def:q}, \eqref{def:h} and Chapman-Kolmogorov equation:
		\begin{align*}
			&\kappa_{ \beta } \int_{ 0 }^{ t } \int_{ D } \pw(s,x,y) y^{ -\alpha - \beta } dy ds \\ 
			&=  \int_{ 0 }^{ t } \int_{ D } \pw(s,x,y) \int_{0}^{\infty} \int_{D} \pw(r,y,z)f'(r)z^{-\gamma - \alpha/2} dz dr dy ds \\ 
			&=  \int_{ 0 }^{ t } \int_{0}^{\infty} \int_{ D } \pw(s + r,x,z) f'(r)z^{-\gamma - \alpha/2} dz dr ds \\
			&=   \int_{0}^{\infty} \int_{ D } [\pw(r,x,z)  - \pw(t + r,x,z)] f(r)z^{-\gamma - \alpha/2} dz dr \\&= x^{ - \beta} - \int_{0}^{\infty} \int_{ D }  \pw(t + r,x,z) f(r)z^{-\gamma - \alpha/2 } dz dr \\
			&= x^{ - \beta} - \int_{0}^{\infty} \int_{ D } \int_{ D } \pw(t,x,y) \pw(r,y,z) f(r)z^{-\gamma - \alpha/2} dy dz dr \\
			&= x^{ - \beta} -  \int_{ D } \pw(t,x,y) y^{ - \beta} dy.
		\end{align*}
	\end{proof}
\noindent Notice that by taking $\beta \rightarrow 0$ in Lemma \ref{lemma:pwIdentity}, we obtain
	\begin{equation} \label{eq:pwIntegral}
		\int_D \pw(t,x,y) dy = 1,\qquad t>0,~x\in D.
	\end{equation}

In the remaining part of the article, we will fix $\delta \in [0, 1/2)$, and put $\kappa=\kappa_\delta$. For the sake of simplicity, let
	\begin{align*}
			q(x) = q_\delta(x)=\kappa x^{-\alpha}.
	\end{align*}
	For $t>0$ and $x, y \in D$ we let $p_0(t,x,y) = p_D(t,x,y)$ and
	\begin{align*} \label{eq:pn1}
		p_n(t,x, y) &= \int_0^t \int_{D} p_D(s,x, z) q(z) p_{n-1}(t-s, z, y)  {\rm d}z  {\rm d}s 
		, \quad n \geq 1.
	\end{align*}  
	We will consider the following Schr\"odinger perturbation of $p_D$ by $q$ 
	\begin{equation*} \label{eq:pwSchPert}
		\tp(t,x,y) = \sum_{n = 0}^{\infty} p_n(t,x,y), \qquad t>0,\; x,y\in D. 
	\end{equation*}
	The above sum is called a perturbation series.
	Obviously $\tp \ge p_D$.  We note that $\tp$ satisfies Duhamel's formula (see e.g. \cite{MR2457489})
	\begin{align}
		\tp(t,x,y) 
		&= p_D(t,x,y) + \int_0^t \int_{D} \tp(s,x,z) q(z) p_D(t-s,z,y) dz ds, \qquad t>0,\; x,y\in D. \label{eq:Df2}
	\end{align}
By considering 	$\pw$ instead of $p_D$, by the same approach (i.e. we use perturbation series), we define weighted kernel $\tpw$, which also satisfies Duhamel's formula
	\begin{align*}
		\tpw(t,x,y) 
		&= \pw(t,x,y) + \int_0^t \int_{D} \tpw(s,x,z) q(z) \pw(t-s,z,y) dz ds. \label{eq:tpwDf2}
	\end{align*}
Analogously to $\pw$, we have
\begin{equation} \label{eq:tpwQuSym}
	\tpw(t,x,y) =\tpw(t,y,x)  \frac{y^{\alpha}}{x^{\alpha}}, \qquad t>0,~x,y \in D.
\end{equation}
By \cite[Theorem~1]{MR3460023}
\begin{equation*}
	\int_{D} \tp(t,x, y)y^{\alpha/2 - \delta}{\rm d}y \leq x^{\alpha/2 - \delta}, \qquad t > 0, x \in D.
\end{equation*}
Hence, $\tp(t,x,y)$ if finite for all $t >0, x >0$ and $a.e.~y \in D.$ 
As a consequence we get
\begin{equation} \label{lemma:superMedian}
	\int_{D} \tpw(t,x, y)y^{-\delta}{\rm d}y \leq x^{-\delta}, \qquad t > 0, x \in D.
\end{equation}

The kernels $\tp$, $\tpw$ has the scaling property (see e.g. \cite[Lemma~2.2]{ MR3933622})
\begin{align*} 
	\begin{cases}
	\tp(t,x, y) = t^{-\frac{1}{\alpha }}\tp \big(1, xt^{-\frac{1}{\alpha }}, yt^{-\frac{1}{\alpha }}\big), \\
	\tpw(t,x, y) = t^{-\frac{1}{\alpha }}\tpw\big(1, xt^{-\frac{1}{\alpha }}, yt^{-\frac{1}{\alpha }}\big),
	\end{cases} t>0,\;~x,y \in D,
\end{align*}
and satisfy Chapman-Kolmogorov equation \cite{MR2457489}:
\begin{align} \label{eq:tpCK}
	\begin{cases}
		\tp(s+t, x, y) = \int_D \tp(s, x, z) \tp(t, z, y) dz, \\
		\tpw(s+t, x, y) = \int_D \tpw(s, x, z) \tpw(t, z, y) dz,
	\end{cases} t,s>0,\;~x,y \in D.
\end{align}

\section{Invariant functions} \label{sec:4}
Equations \eqref{eq:tpInt}, \eqref{eq:tpInvrFun} below play a crucial role in proving estimates \eqref{eq:mainThmEst} (see. \cite[Theorem 3.1]{MR3933622}). The general concept of the proof is based on \cite{MR3933622}, however, the proof of Theorem \ref{thm:tpInt} follows a slightly different approach. We use also the idea presented in  \cite{jakubowski2023relativistic}.
\begin{theorem} \label{thm:tpInt}
	Let $\delta \in (0,1/2]$ and $\beta \in (0,1-\delta) \setminus \{ \delta \}.$ For $t>0$, $x \in D$, 
	\begin{equation} \label{eq:tpInt}
		\int_{ D } \tpw(t,x,y)y^{  - \beta } dy = x^{ - \beta} + (\kappa - \kappa_{ \beta } ) \int_{0}^{t} \int_{ D } \tpw(s,x,y) y^{ -\alpha - \beta }dyds.
	\end{equation}
	Moreover 
	\begin{equation} \label{eq:tpInvrFun}
		\int_D \tpw(t,x,y)y^{- \delta}dy = x^{ - \delta}, \qquad x \in D,~t>0. 
	\end{equation}
\end{theorem}
Theorem \ref{thm:tpInt} states that the function $x^{ - \delta}$ is invariant for $\tpw$. Since $\kappa_{\beta}$ is increasing on (0,1/2) and  decreasing on (1/2,1), by \eqref{eq:tpInt} the function $x^{ - \beta}$ is subharmonic when $0 < \beta \leq \delta$ and superharmonic if $\beta \in (\delta, 1 - \delta)$.  We start by proving a few auxiliary lemmas.

\begin{lemma} \label{lemma:pwIntLowBound}
	Let $R > 1$. There is a constant $c_0$ such that for $s \in (0,1)$ and $z < R$, 
	\begin{equation*} \label{eq:pwIntLowBound}
		\int_{ 0 }^{ 2R } \pw(s,z,y)dy \geq c_0.
	\end{equation*}
\end{lemma}

\begin{proof}
	If $z < s^{ \frac{1}{\alpha} }$, then
	\begin{align*}
		\int_{ 0 }^{ 2R } \pw(s,z,y) dy &\geq c \int_{ 0 }^{ s^{1/\alpha} } \bigg( 1 \wedge \frac{ z^{\alpha/2 } }{ s^{1/2} } \bigg) \bigg( 1 \wedge \frac{ y^{\alpha/2 } }{ s^{1/2} } \bigg) \bigg( s^{-1/\alpha} \wedge \frac{ s }{ |z-y|^{1+\alpha} } \bigg) \frac{ y^{\alpha/2} }{ z^{\alpha/2} } dy \\ 
		&= c \int_{ 0 }^{ s^{1/\alpha} } \frac{ z^{\alpha/2 } }{ s^{1/2} } \frac{ y^{\alpha/2 } }{ s^{1/2} }s^{-1/\alpha} \frac{ y^{\alpha/2} }{ z^{\alpha/2} } dy = \frac{c}{\alpha + 1}.
	\end{align*}
	For $z \geq s^{ \frac{1}{\alpha} }$,
	\begin{align*}
		\int_{ 0 }^{ 2R } \pw(s,z,y)dy &\geq 
		c\int_{ z  }^{ z + s^{1/\alpha} } \bigg( s^{-1/\alpha} \wedge \frac{ s }{ |z-y|^{1+\alpha} } \bigg) \frac{ y^{\alpha/2} }{ z^{\alpha/2} } dy \\ 
		&\geq c\int_{ z  }^{ z + s^{1/\alpha} } s^{-1/\alpha} dy = c.
	\end{align*}
\end{proof}


\begin{lemma} \label{lemma:tpIntApprox2}
	For $R > 0$ and $t > 0$,
	\begin{align} \label{eq:tpIntApprox2}
		\begin{split}
			\big( 1 - tq(R) \big)\tpw(t,x,y) &\le \pw(t,x,y) \\ &+  \int_0^t \int_{z < R} \tpw(t-s,x,z) q(z) \pw(s,z,y) dz ds, \qquad x, y \in D.
		\end{split}
	\end{align}
\end{lemma}

\begin{proof}
	We have $q(z) < q(R)$ for $z > R$.  Therefore by Chapman-Kolmogorov equation 
	\begin{align*}
		&\int_0^{t} \int_{z \ge R} \tpw(t-s,x,z) q(z) \pw(s,z,y) dz ds \\ 
		&\le  q(R) \int_0^{t} \int_{D} \tpw(t-s,x,z) \tpw(s,z,y) dz ds \le t q(R) \tpw(t,x,y).
	\end{align*}
	Hence Duhamel's formula yields the result.
\end{proof}


\begin{lemma} \label{lemma:tpIntApprox}
	For $\delta \in (0, 1/2] $ there exist constant $M$ such that for $t \in (0, \infty)$
	\begin{equation*} \label{eq:tpIntApprox}
		\int_D \tpw(t,x,y)dy \leq M (1 + t^{\delta/\alpha}x^{ - \delta }).
	\end{equation*}
\end{lemma}
\begin{proof}
	By scaling property of $\tpw$ it suffices to consider $t = 1$. Let $R \geq 2^{1/\alpha}.$ Note that by \eqref{eq:kappabeta}, $\big(1 - \frac{\kappa_{ \delta }}{R^\alpha} \big)^{-1} \leq 2 $  since $R^\alpha \geq 2$ and $\kappa_\delta \le1$.  Hence by Lemma \ref{lemma:tpIntApprox2}
	\begin{equation} \label{eq:approx_tg_y^alpha/2_2}
		\tpw(1,x,y) \leq 2 \Bigg[ \pw(1,x,y)  + \int_{ 0 }^{ 1 } \int_{ 0 }^{ R } \tpw(s,x,z)q(z)\pw(1-s,z,y)dzds \Bigg].
	\end{equation}
	Integrating over $D$ we obtain
	\begin{equation*}
		\int_{ D } \tpw(1,x,y) dy \leq 2 \Bigg[ 1  + \int_{ 0 }^{ 1 } \int_{ 0 }^{ R } \tpw(s,x,z)q(z) dzds \Bigg].
	\end{equation*}
	By Lemma \ref{lemma:pwIntLowBound}, inequality \ref{lemma:superMedian} and Fubini theorem,
	\begin{align*}
		& \int_{ 0 }^{ 1 } \int_{ 0 }^{ R } \tpw(s,x,z)q(z) dzds \\ 
		&\leq \frac{1}{c} \int_{ 0 }^{ 1 } \int_{ 0 }^{ R } \tpw(s,x,z)q(z) \int_{0}^{2R} \pw(1-s,z,y) dydzds \\ 
		&\leq  \frac{1}{c}\int_{0}^{2R}  \tpw(1,x,y) dy \leq \frac{(2R)^{\delta}}{c} \int_{0}^{2R}  \tpw(1,x,y)y^{ - \delta} dy \leq \frac{(2R)^{\delta}}{c} x^{ -\delta}.
	\end{align*}
\end{proof}

\begin{lemma}\label{lem:iiptconv}
	Let $\delta \in (0,1/2)$ and $\beta \in (0,1-\delta)$. Then, for every $t>0$ and $x\in D$
	\begin{align}
		&\int_0^t \int_D \tpw(s,x,y) y^{-\beta - \alpha } dy ds < \infty,  \label{EQ1:estbgd}\\
		&\int_D \tpw(t,x,y) y^{-\beta } dy < \infty. \label{EQ2:estbgd}
	\end{align}
\end{lemma}
\begin{proof}
	Following the proof of \cite[Corollary 3.5]{MR3933622}, for $\beta \in (\delta, 1-\delta)$ we have 
	\begin{align}\label{eq:estbgd}
		\int_0^t\int_D \tpw(s,x,z)z^{-\beta-\alpha} dz ds \le \frac{\kappa_\beta}{\kappa_\beta - \kappa} x^{-\beta}.
	\end{align}
	Now let $\beta \in (0,\delta]$. Note that $\beta < 1/2$ and $1/2\in (\delta,1-\delta)$. By Lemma \ref{lemma:tpIntApprox} and \eqref{eq:estbgd},
	\begin{align} \label{eq:pw(-beta-alpha)IntApprox}
		\begin{split}
			\int_0^t\int_D \tpw(s,x,z)z^{-\beta-\alpha} dz ds 
			&\le \int_0^t\int_D \tpw(s,x,z)(1+ z^{-1/2-\alpha})  dz ds \\
			&\le M\int_0^t (1+s^{\delta/\alpha} x^{-\delta}) ds + \frac{\kappa_{1/2}}{\kappa_{1/2} - \kappa} x^{-\beta} <\infty.
		\end{split}
	\end{align}
	Hence, \eqref{EQ1:estbgd} follows by \eqref{eq:estbgd} and \eqref{eq:pw(-beta-alpha)IntApprox}. By Lemma \ref{lemma:pwIdentity} we have 
	\begin{equation} \label{eq:pwy-betaApprox}
		\int_D \pw(t,x,y) y^{-\beta} dy \leq x^{-\beta}.
	\end{equation}
	Therefore, by \eqref{eq:Df2}, \eqref{eq:pwy-betaApprox} and \eqref{EQ1:estbgd}: 
	\begin{align*}
		\int_D \tpw(t,x,y) y^{-\beta } dy \le  \int_D \pw(t,x,y) y^{-\beta } dy + \int_0^t \int_D \tpw(s,x,z) q(z)z^{-\beta} dz ds < \infty, 
	\end{align*}
	which ends the proof of \eqref{EQ2:estbgd}.
\end{proof} 

\begin{proof}[Proof of Theorem \ref{thm:tpInt}]
	Let $t >0$ and $x \in D$. By \eqref{eq:Df2} and Lemma \ref{lemma:pwIdentity}, we have
	\begin{align}
		&\int_D \tpw(t,x,y) y^{-\beta} dy \notag\\
		&=  \int_D \pw(t,x,y) y^{-\beta} dy  + \int_0^t \int_D \tpw(u,x,z) q(z) \left(\int_D \pw(t-u, z, y) y^{- \beta} dy\right) du  dz \label{eq1:thm:tpInt}
	\end{align}
	and 
	\begin{align}
		&\kappa_{ \beta } \int_0^t \int_D \tpw(s,x,y) y^{-\alpha - \beta } dy ds \notag\\
		&= \kappa_{ \beta } \int_0^t \int_D \pw(s,x,y) y^{-\alpha - \beta } dy ds \notag\\ 
		&+  \kappa_{ \beta }  \int_0^t \int_D  \int_0^s \int_D \tpw(u,x,z) q(z) \pw(s-u, z, y) y^{-\alpha - \beta} dz du dy ds \notag\\ 
		&=x^{-\beta} - \int_D \pw(t,x,y) y^{-\beta} dy \notag\\
		&+    \int_0^t \int_D  \tpw(u,x,z) q(z) \left(\kappa_{ \beta }\int_0^{t-u} \int_D  \pw(s, z, y) y^{-\alpha - \beta}  dy ds\right) dz du. \label{eq2:thm:tpInt}
	\end{align}
	We add \eqref{eq1:thm:tpInt} and \eqref{eq2:thm:tpInt} and apply Lemma \ref{lemma:pwIdentity}  to get
	\begin{align}\label{eq:tpIntpos}
		&\int_D \tpw(t,x,y) y^{-\beta} dy  + \kappa_{ \beta } \int_0^t \int_D \tpw(s,x,y) y^{-\alpha - \beta } dy ds \notag\\
		&=  x^{-\beta} + \kappa \int_0^t \int_D \tpw(s,x,y) y^{-\alpha - \beta } dy ds
	\end{align}
	
	Let $\delta \in (0,1/2)$. Then, \eqref{eq:tpInt} follows by \eqref{eq:tpIntpos} and Lemma \ref{lem:iiptconv}. Furthermore, let $\beta \nearrow\delta$. By \eqref{eq:tpInt}, 
\eqref{eq:pw(-beta-alpha)IntApprox} and Lebesgue convergence theorem we get \eqref{eq:tpInvrFun}.
	
	Now, consider $\delta = 1/2$ and let $\beta \in (0,\delta)$. Denote by $\tp_\beta(t,x,y)$ the perturbation of $p_D(t,x,y)$ by $\kappa_\beta x^{-\alpha}$. Put $\tpw_\beta(t,x,y) = x^{-\alpha/2} \tp_\beta(t,x,y)y^{\alpha/2}$. Then, $\tpw$ is the perturbation of $\tpw_\beta$ by $(\kappa - \kappa_\beta)x^{-\alpha}$ (see e.g. \cite{MR2457489}). Hence, by Duhamel formula, Fubini theorem and \eqref{eq:tpInvrFun} applied to $\tpw_\beta$, we get
	\begin{align*}
		(\kappa-\kappa_\beta) \int_0^t \int_D \tpw(s,x,y) y^{-\alpha - \beta } dy ds 
		&= \int_D\int_0^t \int_D \tpw(s,x,y) (\kappa-\kappa_\beta)y^{-\alpha} \tpw_\beta(t-s,y,z) z^{-\beta} dz dy ds \\
		&= \int_D \left(\tpw(t,x,z) - \tpw_\beta(t,x,z)\right) z^{-\beta} dz \\
		&= \int_D \tpw(t,x,z) z^{-\beta} dz - x^{-\beta}, 
	\end{align*}
	which gives \eqref{eq:tpInt}. Next, let $\beta < \delta$. Clearly, $y^{-\beta} < 1+ y^{-\delta}$. By \eqref{lemma:superMedian} and Lemma \ref{lemma:tpIntApprox} we have
\begin{align*}
\int_D \tpw(s,x,y) y^{- \beta } dy \le \int_D \tpw(s,x,y) (1+ y^{-\delta}) dy <\infty.
\end{align*}
By \eqref{eq:tpInt},
\begin{align*}
\int_D \tpw(s,x,y) y^{- \beta } dy \ge x^{-\beta}.
\end{align*}
Hence, taking $\beta \nearrow \delta$, by Lebesgue convergence theorem we get 
\begin{align*}
\int_D \tpw(s,x,y) y^{- \delta } dy \ge x^{-\delta},
\end{align*}
which together with \eqref{lemma:superMedian} gives \eqref{eq:tpInvrFun}.

\end{proof}

\section{Two-Sided Estimates and Joint Continuity of $\tpw(t,x,y)$}
In this section we prove Theorem \ref{thm:main}. Generally we follow the idea of the proof of \cite{MR3933622}*{Theorem 3.1}, however due to the lack of symmetry of $\pw$ we had to make some adjustments. 
We define (see. e.g. \cite{10.1214/18-EJP133}*{Example 3.4})
\begin{equation*}
	\eta_t(y) = \lim_{x \rightarrow 0} \pw(t,x,y), \qquad t>0,~y \in D.
\end{equation*}

By definition of $\pw$ and \eqref{eq:pDApprox} the following estimates for $\eta_t(y)$ holds (see also \cite{10.1214/18-EJP133}): 
\begin{equation} \label{eq:etaApprox}
	\eta_t(y) \approx \frac{y^{\alpha/2}}{t^{1 + 1/\alpha}}   \wedge \frac{t^{1/2}}{ y^{1 + \alpha} }, \qquad t>0,~y \in D.
\end{equation}
For $x \in D$ and $t > 0$ let $H(t,x) = t^{\delta/\alpha}x^{-\delta} + 1$. To simplify we denote $H(x) = H(1, x).$ As a result of equation \eqref{eq:tpInvrFun} and inequality \eqref{eq:tpIntApprox} for $t >0, x \in D$ we get
\begin{equation} \label{eq:tpHIntIneq}
		\int_{D} \tpw(t,x,y) H(t,y)dy \leq (M+1)H(t,x).
\end{equation}	

The following corollary is a consequence of inequalities \eqref{eq:tpHIntIneq} and \eqref{eq:pw(-beta-alpha)IntApprox}.
\begin{corollary} \label{lemma:(tpz-beta)IntApprox}
	If $\beta \in (0, \delta]$, then
	\begin{equation*} \label{eq:(tpz-beta)IntApprox}
		\int_{D} \tpw(1,x,z) z^{-\beta} dz \le CH(x),
	\end{equation*}
	and
	\begin{equation} \label{eq:doubleInttpwApprox}
		\int_{0}^{1} \int_D \tpw(s,x,z) z^{-\alpha - \beta} dz ds \leq cH(x).
	\end{equation}
\end{corollary}

\begin{lemma} \label{lemma:tpxSmallyLargeIntApprox}
	There exists a constant $c > 0$ such that for $t \in (0,1)$, 
	\begin{equation*}
		\tpw(t,x,y) \le c t^{-1/2} \pw(t,x,y) H(t,x), \qquad x \le 2(2\kappa)^{1/\alpha} \le y.
	\end{equation*}
\end{lemma}

\begin{proof}
	Let $R = (2\kappa)^{1/\alpha}$. Due to Lemma \ref{lemma:tpIntApprox2} and since $tq(R) \leq 1/2$ it suffices to estimate the integral in \eqref{eq:tpIntApprox2}. 
	Notice that by integrating \eqref{eq:Df2} over $y$ and by Lemma \ref{lemma:tpIntApprox},
	\begin{align} \label{eq:tpRIntApprox3}
		\int_0^t \int_{z < R} \tpw(t-s,x,z) q(z) dz ds \leq \int_D \tpw(t,x,z) dz \le M H(t,x).
	\end{align}
	Therefore by \eqref{eq:pwUsefullApprox}
	\begin{align*}
		\int_0^t \int_{z < R} \tpw(t-s,x,z) q(z) \pw(s,z,y) dz ds  &\leq c \int_0^t \int_{z < R} \tpw(t-s,x,z) q(z) t^{1/2} y^{-1-\alpha/2} dz ds \\
		&\leq c t^{1/2} H(t,x) y^{-1-\alpha/2} .
	\end{align*}
There exists a constant $c$ such that $ty^{-1 - \alpha} \leq cx^{-\alpha/2} (1 \wedge x^{\alpha/2} t^{-1/2} ) ( t^{-1/\alpha} \wedge t / |x-y|^{-1-\alpha} ). $ Hence $ t^{1/2} H(t,x) y^{-1-\alpha/2} \leq ct^{-1/2}H(t,x) \pw(t,x,y). $
\end{proof}

\begin{lemma} \label{lemma:tpxLargeyLargeIntApprox}
	There exists a constant $c > 0$ such that,
	\begin{equation*}
		\tpw(1,x,y) \le c \pw(1,x,y) H(x), \qquad x, y \ge 2(2\kappa)^{1/\alpha}.
	\end{equation*}
\end{lemma}

\begin{proof}
	Let $R = (2\kappa)^{1/\alpha}$. By Lemma \ref{lemma:tpxSmallyLargeIntApprox} and \eqref{eq:pwUsefullApprox},
	\begin{align*}
		&\int_0^1 \int_{z < R} \tpw(1-s,x,z) q(z) \pw(s,z,y) dz ds  \\
		&= \int_0^1 \int_{z < R} \tpw(1-s,z,x) \frac{ z^{\alpha} }{ x^{\alpha} } q(z) \pw(s,z,y) dz ds  \\
		&\leq c \int_0^1 \int_{z < R} \pw(1-s,z,x) x^{-\alpha} (1-s)^{-1/2} H(1-s,z) \pw(s,z,y) dz ds  \\			
		&\leq c \int_0^1 \int_{z < R}  x^{-1 - 3\alpha/2} (1 + z^{-\delta}(1-s)^{\delta/\alpha} )  \frac{ s^{1/2} }{ y^{1 + \alpha/2} } dz ds \leq c x^{-1 - 3\alpha/2}y^{- 1 - \alpha/2}.
	\end{align*}
	Since $(xy)^{-1-\alpha} > c(1 \wedge |x-y|^{-1-\alpha})$ for $x,y \leq 2R$ we get that $x^{-1 - 3\alpha/2}y^{- 1 - \alpha/2} \leq c \pw(1,x,y)$. Hence Lemma \ref{lemma:tpIntApprox2} yields result.
\end{proof}

\noindent We define $v(y)=y^\alpha/(2^{\alpha+1}\kappa)$. 


\begin{lemma} \label{lemma:tpxSmallySmallInequalities1}
	Let $x, y \leq 2(2\kappa)^{1/\alpha}.$ There exists a constant $c$ such that 
	\begin{align} \label{eq:tpxSmallySmallInequalities1}
		\begin{split}
			\tpw(1,x,y) &\le c H(x) \eta_1(y) + c \int_0^{1/2} \int_{2z < y} \tpw(1-s,x,z)q(z) \eta_s(y) dz ds \\
			&+ c \int_{v(y)}^{1} \int_{2z > y} \tpw(1-s,x,z)q(z)\pw(s,z,y) dzds.
		\end{split}
	\end{align}
\end{lemma}

\begin{proof}
	At first let us notice that by \eqref{eq:pwUsefullApprox} and \eqref{eq:etaApprox} it follows that $\pw(s,z,y) \leq c\eta_s(y)$ for $2z \leq y.$ Additionally for any $z,y \in D$ and $s > 1/2$ we obtain $\pw(s,z,y) \leq c\eta_1(y).$
	
	Let $R = (2\kappa)^{1/\alpha}$. By Lemma \ref{lemma:tpIntApprox2}
	\begin{align*}
		\tpw(1,x,y) &\leq 2\eta_1(y) \\ 
		&+ 2 \bigg( \int_0^1 \int_{2z \le y} \tpw(1-s,x,z)q(z)\pw(s,z,y) dz ds \\ 
		&+ \int_0^1 \int_{2z > y} \tpw(1-s,x,z)q(z)\pw(s,z,y) dz ds \bigg).
	\end{align*} 
	By inequality \eqref{eq:tpRIntApprox3} 
	\begin{align*}
		&\int_0^1 \int_{2z \le y} \tpw(1-s,x,z)q(z)\pw(s,z,y) dz ds \\ 
		&\leq c\int_0^{1/2} \int_{2z \le y} \tpw(1-s,x,z)q(z)\eta_s(y) dz ds + cH(x)\eta_1(y).
	\end{align*}
	We have $v(y) \leq 1$ whenever $y \leq 2R.$ By inequality $\pw \leq \tpw$ and Chapman-Kolmogorov equation
	\begin{align*}
		&\int_0^1 \int_{2z > y} \tpw(1-s,x,z)q(z)\pw(s,z,y) dz ds \\ 
		&\leq \int_0^{1 \wedge v(y)} \int_{2z > y} \tpw(1-s,x,z)q(y/2) \pw(s,z,y) dz ds \\
		&+\int_{1 \wedge v(y)}^1 \int_{2z > y} \tpw(1-s,x,z)q(z) \pw(s,z,y) dz ds \\ 
		&\leq \frac{1}{2} \tpw(1,x,y) + \int_{1 \wedge v(y)}^1 \int_{2z > y} \tpw(1-s,x,z)q(z) \pw(s,z,y) dz ds.
	\end{align*} 
	Which ends proof of the inequality \eqref{eq:tpxSmallySmallInequalities1}. 
\end{proof}

\begin{lemma} \label{lemma:tpxSmallySmallInequalities2}
	For every $\gamma \in (0, \delta)$ there is a constant $C_{\gamma}$ such that 
	\begin{equation} \label{eq:tpxSmallySmallInequalities2}
		\tpw(1,x,y) \leq C_{\gamma} y^{\gamma - 1}H(x), \qquad x, y \leq 2(2\kappa)^{1/\alpha}.
	\end{equation}
\end{lemma}

\begin{proof}
		Let $\gamma \in (0, \delta).$ By \eqref{eq:doubleInttpwApprox} and since $\eta_s(y) \leq cy^{-1}$ we have,
	\begin{align*}
		\int^{1/2}_0\int_{2z \leq y}\tpw(1-s,x,z)q(z)\eta_s(y)dzds 
		&\le c\int_0^1 \int_{z \le y} \tpw(1-s,x,z)y^\gamma z^{-\gamma-\alpha} y^{-1} dz ds \\
		&\le c y^{\gamma-1} H(x).
	\end{align*}
	If $v(y) \le s \le 1$, then $\pw(s,z,y) \le cy^{-1}$ for $z < 2y$ or $z>2y$. Therefore
	\begin{align*}
		&\int^1_{v(y)}\int_{2z>y} \tpw(1-s,x,z)q(z)\pw(s,z,y)dzds \\
		&\le cy^{-1}\int_{v(y)}^1 \int_{z<2y} \tpw(1-s,x,z) q(z) dz ds  +\,c\int_{v(y)}^1 \int_{z> 2y} \tpw(1-s,x,z) q(z) z^{-1} dz ds \\ 
		&\le \, c y^{\gamma-1}\int_{0}^1 \int_{D} \tpw(1-s,x,z) z^{-\gamma-\alpha} dz ds \le  c H(x)y^{\gamma-1},
	\end{align*}
	where the last inequality follows from \eqref{eq:doubleInttpwApprox}. Since $\eta_1(y)\leq c y^{\gamma-1}$, we obtain the claim by inequality \eqref{eq:tpxSmallySmallInequalities1}.
\end{proof}


	\begin{lemma}\label{lemma:Inttp(z_-beta)Approx} 
	For all $\beta \in (\delta, 1 + \alpha - \delta)$ and $\lambda \in (0, \delta)$ there is a constant $C_\lambda$ such that
	\begin{equation*}
		\int_{D} \tpw(t,x,z) z^{-\beta} dz \le C_\lambda x^{-\beta} \Big(1 \vee x^{-\lambda} \Big), \qquad 0< t \le 1 , \quad x \in D.
	\end{equation*}
	\end{lemma}

	\begin{proof}
	Let $\gamma = \delta-\lambda$, so that $\gamma \in (0, \delta)$ and $\beta - \gamma >0$. Fix $R=2(2\kappa)^{1/\alpha}$ and let $x \le R$. By \eqref{eq:tpwQuSym} and by inequality \eqref{eq:tpxSmallySmallInequalities2} we have,
	\begin{align*}
		\int_{z \le R} \tpw(1,x,z) z^{-\beta} dz &\le  c\int_{z\le y} x^{\gamma-1} z^{-\delta} \bigg(\frac{z}{x}\bigg)^{\alpha} z^{-\beta}  dz
		+ c \int_{x < z \le R} z^{\gamma-1} x^{-\delta}  z^{-\beta} \frac{z^{\alpha}}{x^{\alpha}} dz\\ &\le c x^{\gamma-\beta-\delta} = c x^{-\beta -\lambda}.
	\end{align*}
	Now let $x > R$. By Lemma \ref{lemma:tpxSmallyLargeIntApprox} and Lemma \ref{lemma:tpxLargeyLargeIntApprox} we have $\tpw(1,x,z) \le \pw(1,x,z)H(z)$ for all $z \in D$. Therefore by \eqref{eq:pwQuSym} and \eqref{eq:pwUsefullApprox}
	\begin{align*}
		&\int_{2z \le x} \tpw(1,x,z) z^{-\beta} dz \le  c\int_{2z \le x} \pw(1,x,z) H(z) z^{-\beta} dz \\ 
		&\le c \int_{2z \le x} x^{-1 - \alpha} z^{-\beta + \alpha}\left(1 + z^{-\delta}\right) dz \le cx^{-\beta}\left(1 + z^{-\delta}\right) \le c x^{-\beta},
	\end{align*}
	and also by \eqref{eq:pwIntegral} and Lemma \ref{lemma:tpxLargeyLargeIntApprox}
	\begin{align*}
		&\int_{2z \ge x} \tpw(1,x,z) z^{-\beta} dz \le  c\int_{2z \ge x} \pw(1,x,z) H(z) x^{-\beta} dz \\ 
		&\le c \int_{D} \pw(1,x,z) x^{-\beta} dz = c x^{-\beta}.
	\end{align*}
		By scaling property of $\tpw$:
	\begin{align*}
		&\int_D \tpw(t,x,z) z^{-\beta} dz = \int_D t^{-\beta/\alpha} \tpw(1, x t^{-1/\alpha}, z) z^{-\beta} dz \\
		&\le c t^{-\beta/\alpha} \big(x t^{-1/\alpha} \big)^{-\beta} \Big(1 \vee \big(x t^{-1/\alpha} \big)^{-\lambda} \Big) \le c x^{-\beta} \Big(1 \vee x^{-\lambda} \Big).
	\end{align*}
	\end{proof}

	For $\beta > 0$ we denote $H_\beta(x) = x^{-\beta} + 1$. By Lemma \ref{lemma:(tpz-beta)IntApprox} and Lemma \ref{lemma:Inttp(z_-beta)Approx}, for all $\beta \in (0, 1 + \alpha-\delta)$ and $\lambda \in (0, \delta)$ there is a constant $C_{\lambda}$ such that
	\begin{align} \label{eq:inttpH}
		\int_{D} \tpw(t,x,y) H_{\beta}(y) dy \le C_\lambda H_{(\beta+\lambda)\vee \delta}(x), \qquad t\le1,\;   x\in D.
	\end{align}

	In the next lemma we improve this result.
	
	\begin{lemma} \label{lemma:inttpH(beta)}
	Let $0 <\beta < 1 + \alpha - \delta$. There is a constant $C_\beta$ such that 
		
	\begin{align*}
		\int_{D} \tpw(t,x,y) H_\beta(y) dy\le C_\beta t^{-\beta/\alpha}H(t^{-1/\alpha}x), \qquad 0 < t \le 1,\;   x\in D.
	\end{align*}
	\end{lemma}

	\begin{proof}
	By the scaling of $\tpw$ it suffices to prove the result for fixed $t>0$.
	If $\beta \le \delta$, then we simply apply Lemma \ref{lemma:tpIntApprox} and Lemma \ref{lemma:(tpz-beta)IntApprox}, so suppose that $\beta = \delta + \xi \alpha$ for some $\xi >0$. 
		
	First, let $\xi <1$. By Lemma \ref{lemma:tpIntApprox}, Theorem \ref{thm:tpInt} and Lemma \ref{lemma:(tpz-beta)IntApprox},
	\begin{align} \label{eq:dInttpHgamma1}
		\begin{split}
		&\int_0^1 \int_{D} \tpw(s,x,z) H_{\gamma}(z) \,dz\,ds \\
		&\le c\int_0^1 \int_{D} \tpw(s,x,z) H_{\gamma \vee (\alpha + \delta/2)}(z) \,dz\,ds \le c H(x),  \qquad \text{if} \quad \delta < \gamma < \delta+\alpha.
		\end{split}
	\end{align}	
	Let $0 < \epsilon < (1-\xi)\alpha\wedge \delta$, so that $\beta + \epsilon < \delta+\alpha$. By \ref{eq:inttpH} and \eqref{eq:dInttpHgamma1}, 
	\begin{align*}
		\int_{D} \tpw(1,x,y) H_{\beta}(y) dy & = \int_{D}\int_{D} \tpw(s,x,z)\tpw(1-s,z,y) H_{\beta}(y)\, dz\,dy \label{eq1.5:inttp} \\
		&=\int_0^1\int_{D}\int_{D} \tpw(s,x,z)\tpw(1-s,z,y) H_{\beta}(y)\, dy\,dz\,ds \notag \\
		&\le C \int_0^1\int_{D} \tpw(s,x,z) H_{\beta+\epsilon}(z) \,dz\,ds  \le C H(x),
	\end{align*}
	as needed.
	Next, let $\xi \ge 1$.  By Lemma \ref{lemma:tpIntApprox} and Theorem \ref{thm:tpInt}
	\begin{align} 
		\int_0^1\int_{D} \tpw(s,x,z) H_{\gamma}(z) \,dz\,ds &\le c H_{\gamma-\alpha}(x) \quad \text{if} \quad \delta +\alpha < \gamma < 1 + \alpha - \delta.
	\end{align} \label{eq:dInttpHgamma2}
	Let us fix $\nu \in (0,1)$  such that $\rho := (1-\nu)\alpha < 1 + \alpha-\beta-\delta$ and $\rho<\delta$.
	By \eqref{eq:inttpH} and \eqref{eq:dInttpHgamma2}, we have for $\delta+\alpha\leq \gamma\leq \beta$,
	\begin{align}
		\int_{D} \tp(1,x,y) H_{\gamma}(y) dy  &=\int_0^1\int_{D}\int_{D} \tp(s,x,z)\tp(1-s,z,y) H_{\gamma}(y)\, dy\,dz\,ds  \label{eq3:inttp} \\
		&\le C \int_0^1\int_{D} \tp(s,x,z) H_{\gamma+\rho}(z) \,dz\,ds  \le C H_{\gamma-\nu\alpha}(x). \notag
	\end{align}
	We choose $n \in \mathbb{N}$ so that $(n-1)\nu +1 \leq  \xi < n \nu +1$.  By scaling property of $\tpw$, \eqref{eq3:inttp} and \eqref{eq1.5:inttp},
	\begin{align*}
		\int_{D} \tpw(1,x,y) H_{\beta}(y) dy &=  \int_{D} \tpw(n + 1, (n + 1)^{1/\alpha}x,y) H_{\beta}((n + 1)^{1/\alpha}y) dy \\
		& \leq c\int_{D} \tpw(n+1,(n + 1)^{1/\alpha}x,y) H_{\beta}(y) dy \\
		& = \int_{D}\ldots\int_{D} \tpw(1,(n + 1)^{1/\alpha}x,z_1) \ldots \tpw(1,z_n,y) H_{\beta}(y)\, dy\, dz_n\ldots dz_1 \\
		& \le C^n\int_{D} \tpw(1,(n + 1)^{1/\alpha}x,z_1) H_{\beta-n\nu\alpha}(z_1)\, dz_1 \\
		& = C^n\int_{D} \tpw(1,(n + 1)^{1/\alpha}x,z_1) H_{\delta +(\xi-n\nu)\alpha}(z_1)\, dz_1 \\
		&\le C^{n+1} H(x),
	\end{align*}
		as needed.
	\end{proof}
	
	\begin{lemma}\label{lemma:tpxSmallySmallInequality2} 
		For each $\rho \in (0,\delta)$ there is  $C_\rho$ such that
		\begin{align*}
			\tpw(1,x,y) \le C_\rho H(x) y^{\alpha - \delta - \rho}, \qquad 0 < x, y \le 2(2\kappa)^{1/\alpha}.
		\end{align*}
	\end{lemma}
	
	\begin{proof}
		We will apply Lemma \ref{lemma:tpxSmallySmallInequalities1}.
		
		By Lemma \ref{lemma:inttpH(beta)}, for $s \in (0,1/2)$ we have
		\begin{align} \label{eq:tpxSmallySmallInequality2Eq1} 
			\int_{z < 2y} \tpw(1-s,x,z) q(z) dz  
			\le \int_{D} \tpw(1-s,x,z) \bigg(\frac{2y}{z}\bigg)^{1 -\delta - \rho} q(z) dz
			\le c y^{1-\delta-\rho} H(x).
		\end{align}
		Therefore by \eqref{eq:etaApprox},
		\begin{align*} 
			\int_0^{1/2} \int_{2z < y} \tpw(1-s,x,z) q(z) \eta_s(y) dzds \le c H(x) y^{1-\delta - \rho} \int_0^{\infty} \eta_s(y) ds   
			\le  c H(x) y^{ \alpha -\delta - \rho} .
		\end{align*}
		By \eqref{eq:pDApprox}, $\pw(s,z,y) \leq c y^{\alpha}$ for $s \in (1/2, 1)$. Hence by \eqref{eq:tpRIntApprox3}
		\begin{equation*}
			\int^1_{1/2}\int_{2z>y} \tpw(1-s,x,z)q(z)\pw(s,z,y)dzds \leq cH(x)y^{\alpha}.
		\end{equation*}
		Recall that $v(y)=y^\alpha/(2^{\alpha+1}\kappa)$. Let $v(y) < 1/2$. Since $(s^{-1/\alpha} \wedge sz^{-1-\alpha}) \leq s^{(\beta - 1)/\alpha}z^{-\beta}$ for $z,s > 0,~\beta \in [0, 1 + \alpha]$ and by \eqref{eq:pDApprox} we have that $\pw(s, z, y) \leq y^{\alpha}s^{-(\alpha + \delta + \rho)/\alpha}z^{1 - \delta - \rho}$ for $z > 2y$ and $\beta = 1 - \delta - \rho$.  Hence by \eqref{eq:pDApprox}, \eqref{eq:tpxSmallySmallInequality2Eq1} and Lemma \ref{lemma:inttpH(beta)}, 
		\begin{align*}
			&\int_{v(y)}^{1/2} \int_{2z > y} \tpw(1-s,x,z)q(z)\pw(s,z,y)dzds \\
			&\le  c\int_{v(y)}^{1/2} \int_{z < 2y} \tpw(1-s,x,z) q(z) y^{\alpha} s^{-(1 + \alpha)/\alpha} dzds \\
			&+ c \int_{v(y)}^{1/2} \int_{z > 2y} \tpw(1-s,x,z) q(z) \frac{y^{\alpha} s^{-(\alpha + \delta + \rho)/\alpha}}{z^{1-\delta-\rho}} dzds \\ 
			&\le c H(x) y^{\alpha + 1 - \delta - \rho} \int_{v(y)}^{1/2} s^{-(1 + \alpha)/\alpha} ds + cH(x)y^{\alpha} \int_{v(y)}^{1/2} s^{-(\alpha + \delta + \rho)/\alpha} ds 
			\le c H(x) y^{\alpha - \delta - \rho}.
		\end{align*}
		Since $H(x)\eta_1(y) \le cH(x)$, we obtain
		$\tpw(1,x,y) \le CH(x)y^{\alpha - \delta - \rho}$. 
	\end{proof}

	\begin{lemma}\label{lemma:tpUpperBound}
		There exists a constant $C$ such that
		\begin{align*}
			\tp(t,x,y) \le C p_D(t,x,y) H(t,x)H(t,y), \qquad x,y \in D.
		\end{align*}
	\end{lemma}
	
	\begin{proof}
		Let $R=2(2\kappa)^{1/\alpha}$. By Lemma \ref{lemma:tpxSmallyLargeIntApprox} and Lemma \ref{lemma:tpxLargeyLargeIntApprox},
		\begin{equation} \label{eq:tpUpperBoundEq1}
			\tp(1, x, y) \leq Cp_D  (1, x, y)H(x)H(y), \qquad x \in D,~y > R.
		\end{equation}
		Due to symmetry,  
		\begin{equation} \label{eq:tpUpperBoundEq2}
			\tp(1, x, y) \leq Cp_D(1, x, y)H(x)H(y), \qquad x >R,~y \in D.
		\end{equation}
		By scaling property of $\tp$ we only have to show that inequality \eqref{eq:tpUpperBoundEq1} holds for $0 < x, y < R$. Note that, by \eqref{eq:tpUpperBoundEq1}, \eqref{eq:tpUpperBoundEq2} and Chapman-Kolmogorov equation
		\begin{align*}
			\tp(1,x,y) &= c\int_D \tp(1, x2^{1/\alpha}, z)\tp(1, z, y2^{1/\alpha}) dz \\
		 	& \leq c\int_{z < R} \tp(1, x2^{1/\alpha}, z)\tp(1, z, y2^{1/\alpha}) dz \\
		 	&+ c\int_{z > R} p_D(1, x2^{1/\alpha}, z) H(x) p_D(1, z, y2^{1/\alpha}) H(y) dz \\
		 	&\leq c\int_{z < R} \tp(1, x2^{1/\alpha}, z)\tp(1, z, y2^{1/\alpha}) dz \\ 
		 	&+ cp_D(2, x2^{1/\alpha}, y2^{1/\alpha})H(x)H(y), \qquad 0 < x, y < R,
		\end{align*}
		therefore it suffices to estimate the integral in the last inequality above. 
		
		Firstly we consider $2\kappa^{1/\alpha} < x, y < R.$ By \eqref{eq:tpUpperBoundEq1} and \eqref{eq:tpUpperBoundEq2}
		\begin{align*}
			&\int_{z < R} \tp(1, x2^{1/\alpha}, z)\tp(1, z, y2^{1/\alpha}) dz \\
			& \leq cH(x)H(y)\int_{z < R} p_D(1, x2^{1/\alpha}, z) p_D(1, z, y2^{1/\alpha})H^2(z)dz \\
			&\leq c p_D(1, x, y) H(x)H(y) \int_{z < R} z^{\alpha} H^2(z) dz \\
			&\leq c p_D(1, x, y) H(x)H(y),
		\end{align*}
		where the second inequality follows from the fact that $p_D(1, x2^{1/\alpha}, z) p_D(1, z, y2^{1/\alpha}) < cz^{\alpha}$ for some positive $c$ and that $p_D(1, x, y) > c_2$ for some $c_2 > 0$ and $2\kappa^{1/\alpha} < x, y < R.$
		
		Now we let $x, y < 2\kappa^{1/\alpha}.$ Fix $\rho \in (0, 1/2 - \delta).$ By Lemma \ref{lemma:tpxSmallySmallInequality2}
		\begin{align*}
			&\int_{z < R} \tp(1, x2^{1/\alpha}, z)\tp(1, z, y2^{1/\alpha}) dz \\
			&\leq c \int_{z < R} H(x)z^{\alpha/2 - \delta - \rho}x^{\alpha/2} H(y) z^{\alpha/2 - \delta - \rho} y^{\alpha/2} dz \\ 
			&\leq cx^{\alpha/2}y^{\alpha/2}H(x)H(y) \leq p_D(1,x,y) H(x) H(y).
		\end{align*}
		As a last case we consider $x < 2\kappa^{1/\alpha} < y < R$, since the analogous case, $y < 2\kappa^{1/\alpha} < x < R$ follows from symmetry. By Lemma \ref{lemma:tpxSmallySmallInequality2} and \eqref{eq:tpUpperBoundEq1} we have
		\begin{align*}
			&\int_{z < R} \tp(1, x2^{1/\alpha}, z)\tp(1, z, y2^{1/\alpha}) dz \\
			&\leq c \int_{z < R} H(x)z^{\alpha/2 - \delta - \rho}x^{\alpha/2} p_D(1,z, y2^{1/\alpha}) H(y) dz \\ 
			&\leq cx^{\alpha/2}H(x)H(y) \int_{z < R} z^{\alpha - 2\delta - \rho} dz \leq cx^{\alpha/2}H(x)H(y).
		\end{align*}
		Since $c_1 \leq (1 \wedge y^{\alpha/2}) p(1,x,y)$ we get $cx^{\alpha/2}H(x)H(y) \leq c p_D(1,x,y) H(x) H(y).$
	\end{proof}
	
	\subsection{Lower bound of heat kernel}\label{sec:Lb}
	By scaling property is suffices to prove that for all $x, y \in D$  $\tp(3,x ,y) \geq cp_D(3,x,y)H(3, x)H(3, y).$  
	Since $H(3,z) \leq 3^{\delta/\alpha + 1}$ for $z \geq 1$, we observe that 
	\begin{align} \label{eq:lowerBound1}
		\tp(3, x, y) \geq c p_D(3,x,y) H(3, x) H(3, y), \qquad x,y \geq 1.
	\end{align}
	Due to symmetry of $\tp$ it suffices to prove \eqref{eq:lowerBound1} for $x < 1$ and $y \in D$. \newline
	Let $\mu(dy) = H^2(y)dy$ and $q(t,x,y) = \tpw(t,x,y)/(H(x)H(y)).$ By \eqref{eq:tpCK} we have, 
	\begin{equation*} \label{eq:qCK}
		q(t + s, x, y) = \int_D q(t, x, z)q(s, z, y) \mu(dz).
	\end{equation*}
	Next by \eqref{eq:pwIntegral} and \eqref{eq:tpIntApprox}
	\begin{equation*}
		1 \leq \int_D q(1,x,y)\mu(dy) \leq M + 1.
	\end{equation*}
	Therefore by Lemma \ref{lemma:tpUpperBound} and \eqref{eq:pwUsefullApprox}, there is $R > 2$ such that for $x \in (0,1)$
	\begin{equation*}
		\int_R^{\infty} q(1,x,y) \mu(dy) \leq c \int_R^{\infty} \pw(1,x,y) dy \leq c\int_R^{\infty} y^{-1-\alpha/2} dy \leq \frac{1}{4}. 
	\end{equation*}
	Furthermore there is $r \in (0, 1)$ such that for all $x \in D$, 
	\begin{align*}
		\int_0^r q(1,x,y) \mu(dy) &\leq c \int_0^r \pw(1,x,y) H^2(y) dy \\
		&\leq c \int_0^r y^{\alpha} H^2(y) dy \leq \frac{1}{4}.
	\end{align*} 
	Hence for $x \in (0,1),$
	\begin{equation*}
		\int_r^{R} q(1,x,y) \mu(dy) \geq \frac{1}{2}.
	\end{equation*} 
	For $x \in (0,1)$ and $y \geq r$ we get
	\begin{align*}
		q(2, x, y) &\geq \int_r^R q(1, x, z) q(1, z, y) \mu(dz) \\
		&\geq c \int_r^R q(1, x, z) \frac{\pw(1,z,y)}{H^2(r)} \mu(dz) \\
		&\geq cy^{\alpha/2}(R + y)^{-1 - \alpha} \int_r^R q(1,x,z) \mu(dz) 
		\geq \frac{c}{2} y^{\alpha/2}(R + y)^{-1 - \alpha},
	\end{align*}
	where $H(r) = H(z)$ for $|z| = r.$ Hence by \eqref{eq:pDApprox},
	\begin{equation} \label{eq:lowerBound2}
		q(2,x,y) \geq \pw(2,x,y).
	\end{equation}
	In the same way we obtain that 
	\begin{align} \label{eq:lowerBound3}
		q(3,x,y) \geq \pw(3,x,y) \qquad x \in (0,1), y \geq r.
	\end{align}
	Let $x, y < 1$. By \eqref{eq:lowerBound1} and \eqref{eq:lowerBound2}
	\begin{align} \label{eq:lowerBound4}
	\begin{split}
		q(3,x,y) &\geq c\int_r^R q(1, x, z)q(2, z, y) \mu(dz) \\
		&= c\int_r^R q(1, x, z)q(2, y, z) \bigg( \frac{y}{z} \bigg)^{\alpha} \mu(dz)  
		\geq c \int_r^R q(1,x,z) y^{\alpha} \mu(dz) \geq c \pw(3,x,y).
	\end{split}
	\end{align}
	By scaling property of $\tp$, \eqref{eq:lowerBound1}, \eqref{eq:lowerBound3} and \eqref{eq:lowerBound4} we have that
	\begin{align} \label{eq:tpLowerBound}
		\tp(t, x, y) \geq c p_D(t,x,y) H(t, x) H(t, y), \qquad x,y \in D,~t > 0.
	\end{align}
	
	\subsection{Continuity of $\tp(t,x,y)$}\label{sec:JC}
	The proof of joint continuity of $\tp(t,x,y)$ is based on proof given in [\cite{MR3933622}, Subsection 4.3]. 
	For $\gamma \in (-\alpha, 1 + \alpha/2)$ we have that (see. \cite{jakubowski2022groundstate}, Lemma 3.1):
	\begin{equation} \label{eq:intpDy(-gamma)Approx}
		\int_D p_D(t,x,y) y^{- \gamma} dy \leq cx^{- \gamma} \left(1 \vee tx^{-\alpha} \right)^{-1/2 - \gamma/\alpha}, \qquad t>0, x,y \in D.		
	\end{equation} 
	\begin{lemma}
		For $x,y \in D$ and $\epsilon < 1/2$ we have that: 
		\begin{equation} \label{eq:Int(tpqpD)Approx1}
			\int_0^{\epsilon} \int_D \tp(1-s,x,z)q(z)p_D(s,z,y) dzds \leq c \epsilon H(x)H(y)y^{-\alpha/2}.
		\end{equation}
	\end{lemma}
	\begin{proof}
		Let $\epsilon < 2^{-1} $. By Lemma \ref{lemma:tpUpperBound}, \eqref{eq:pDApprox} and inequality \eqref{eq:intpDy(-gamma)Approx} we have, 
		\begin{align*}
			&\int_0^{\epsilon} \int_D \tp(1-s,x,z)q(z)p_D(s,z,y) dzds \\
			&\leq c\int_0^{\epsilon} \int_D p_D(1-s,x,z)H(x)H(z)q(z)p_D(s,z,y) dzds \\
			&\leq cH(x) \int_0^{\epsilon} \int_D p_D(s,y,z)z^{\alpha/2} H(z)q(z) dzds \\
			&\leq c\epsilon H(x) y^{-\alpha/2}\left(1 + y^{-\delta} \right) = c\epsilon H(x)H(y)y^{-\alpha/2}.
		\end{align*}
	\end{proof}
	
	Let $x, y \in D$. In a similar manner, we get,
	\begin{align} \label{eq:Int(tpqpD)Approx2}
		&\int_{1 - \epsilon}^{1} \int_D \tp(1-s,x,z)q(z)p_D(s,z,y) dzds \leq c\epsilon H^2(x)x^{-\alpha/2}, \qquad x, y \in D.
	\end{align}

	\begin{lemma}\label{lemma:tpConti1}
		For every $x \in D$, the function $D \ni y\mapsto \tp(1,x,\cdot)$ is continuous. 
	\end{lemma}
	\begin{proof}
		Fix $x,y \in D$ and let $w \to y$. By \eqref{eq:Df2}
		\begin{align*}
			\tp(1,x,y) - \tp(1,x,w) &= p_D(1,x,y) - p_D(1,x,w) \\
			&+ \int_0^1 \int_{D} \tp(1-s,x,z) q(z) [p_D(s,z,y) - p_D(s, z,w)]\,dz\,ds. 
		\end{align*}
		Since $p_D$ is jointly continuous it suffices to prove that integral converges to $0$.
		Let $\epsilon < 2^{-1}$. Obviously $|p_D(s,z,y) - p_D(s, z,w)| < p_D(s,z,y) + p_D(s, z,w)$. Therefore by inequality \eqref{eq:Int(tpqpD)Approx1}
		\begin{align*} 
			\begin{split}
				&\int_0^\epsilon \int_{D} \tp(1-s,x,z) q(z) [p_D(s,z,y) - p_D(s,z,w)]\,dz\,ds \\
				&\le c \epsilon H(x) \left(H(y)y^{-\alpha} + H(w)w^{-\alpha} \right).
			\end{split}
		\end{align*}
		For $s \in (\epsilon,1)$ and $w$ closely to $y$ we have $p_D(s,z,w) \approx p_D(s,z,y)$. By the dominated convergence theorem and continuity of $p_D$, 
		\begin{align*}
			\int_\epsilon^1 \int_{D} \tp(1-s,x,z) q(z) [p_D(s,z,y) - p_D(s, z,w)]\,dz\,ds \to 0 \quad \mbox{ as } \quad w \to y.
		\end{align*}
		This proves the desired assertion. 
	\end{proof}
	
	\begin{lemma} \label{lemma:tpConti2}
		The function $\tp(t,x,y)$ is jointly continuous in $t > 0$ and $x,y \in D.$
	\end{lemma}
	
	\begin{proof}
		Fix $x, y \in D$. Let $\hat{x} \to x$, $\hat{y} \to y$ and $\epsilon < 2^{-1}$. By \eqref{eq:Int(tpqpD)Approx1} and \eqref{eq:Int(tpqpD)Approx2},
		\begin{align*}
			&\int_0^1 \int_D |\tp(1-s,\hat{x}, z)p_D(s,z,\hat{y}) - \tp(1-s, x, z)p_D(s,z,y)|q(z)dzds \\
			&\leq \int_0^{\epsilon} \int_D [\tp(1-s,\hat{x}, z)p_D(s,z,\hat{y}) + \tp(1-s, x, z)p_D(s,z,y)]q(z)dzds \\
			&+ \int_{1-\epsilon}^{1} \int_D [\tp(1-s,\hat{x}, z)p_D(s,z,\hat{y}) + \tp(1-s, x, z)p_D(s,z,y)]q(z)dzds \\
			&+ \int_{\epsilon}^{1 - \epsilon} \int_D |\tp(1-s,\hat{x}, z)p_D(s,z,\hat{y}) - \tp(1-s, x, z)p_D(s,z,y)|q(z)dzds \\
			&\leq \epsilon C(x,y) + \int_{\epsilon}^{1 - \epsilon} \int_D |\tp(1-s,\hat{x}, z)p_D(s,z,\hat{y}) - \tp(1-s, x, z)p_D(s,z,y)|q(z)dzds,
		\end{align*}
		where $C(x,y)$ is a constant. Lemma \ref{lemma:tpUpperBound} and \eqref{eq:tpLowerBound} yields that $\tp(1-s, \hat{x}, z) \approx \tp(1-s, x, z)$ and $p_D(1-s, z, \hat{y}) \approx p_D(1-s, z, y)$ for $\epsilon < s < 1 - \epsilon$. By Lemma \ref{lemma:tpConti1} and the dominated convergence theorem, the last integral is arbitrarily small.   
	\end{proof}
	
	\begin{proof}[Proof of Theorem \ref{thm:main}] Lemmas \ref{lemma:tpUpperBound} and  \ref{eq:pDApprox}, inequality \eqref{eq:tpLowerBound} and Lemma \ref{lemma:tpConti2} yield demanded result. 
	\end{proof}
	
	\subsection{Blowup}\label{sec:BU} 
	As a corollary of Theorem \ref{thm:main} we can get the following result
	\begin{proposal}
		If $\kappa > \kappa_{1/2} = \kappa^{*}$, then $\tp(t,x,y) \equiv \infty.$ 
	\end{proposal}
	
	\begin{proof}
		Denote by $\tp^{*}$ the perturbation of $p_D$ by $\kappa^*x^{-\alpha}$. By using Theorem \ref{thm:main}, we get 
		\begin{equation} \label{eq:tp*Approx}
			\tp^*(t,x,z) \geq c \left( 1 \wedge \frac{x^{\alpha}}{t} \right)^{(\alpha - 1)/(2\alpha)} \frac{z^{(\alpha - 1)/2}}{t^{(\alpha - 1)/(2\alpha)}}  \left( t^{-1/\alpha}\wedge \frac{t}{|x-z|^{1+\alpha}} \right), \qquad z < t^{1/\alpha}.
		\end{equation}
		The kernel $\tp$ may be considered as a perturbation of $\tp^{*}$ by $(\kappa-\kappa^*)x^{-\alpha}$ (see \cite{MR2457489}). By Duhamel's formula and \eqref{eq:tp*Approx}, 
		\begin{align*}
			\tp(t,x,y) &= \tp^*(t,x,y) + (\kappa-\kappa^*)\int_0^t \int_D \tp(t-s,x,z)z^{-\alpha}\tp^*(s,x,z)dz ds \\ 
			&\geq (\kappa-\kappa^*)\int_0^t \int_D \tp^*(t-s,x,z)z^{-\alpha}\tp^*(s,x,z)dz ds = \infty.
		\end{align*}
	\end{proof}
	
	\subsection*{Acknowledgments}
We thank Krzysztof Bogdan, Rupert Frank and Konstantin Merz for helpful comments on the paper.
	\bibliographystyle{abbrv}
	\bibliography{shlbib}
		
\end{document}